\documentclass[oneside, 10pt]{amsart}

\usepackage{amsmath, amsfonts, amssymb, amsthm}

\newtheorem{theorem}{Theorem}[section]
\newtheorem{lemma}[theorem]{Lemma}

\theoremstyle{definition}

\newtheorem{remark}[theorem]{Remark}

\numberwithin{equation}{section}

\begin{document}

\title{A sharp constant for the Bergman  projection}

\author{Marijan Markovi\'{c}}

\address{
Faculty of Natural Sciences and Mathematics\endgraf
University of Montenegro\endgraf
Cetinjski put b.b.\endgraf
81000 Podgorica\endgraf
Montenegro}

\email{marijanmmarkovic@gmail.com}



\begin{abstract}
For the Bergman projection operator $P$  we prove that
\begin{equation*}
\|P\|_{{L^1(B,d\lambda)\rightarrow B_1}}= \frac {\left(2n+1\right)!}{n!}.
\end{equation*}
Here     $\lambda$ stands for the invariant     metric in the unit ball $B$ of
$\mathbf{C}^n$,        and  $B_1$  denotes    the Besov space with an adequate
semi--norm. We also consider a generalization of this result. This generalizes
some  recent  results due to Per\"{a}l\"{a}.
\end{abstract}

\maketitle

\section{Introduction and the main result}

%% THE BERGMAN PROJECTION
This paper  deals with the  Bergman projection operator $P_\sigma$, which is
an  integral operator with the kernel
\begin{equation*}
K_\sigma(z,w)
=  \frac {\left(1-|w|^2\right)^\sigma }
{\left(1-\left<z,w\right>\right)^{n+1+\sigma}},
\end{equation*}
i.e.,
\begin{equation*}
P_\sigma f (z) \,    = \int_B\,  K_\sigma (z,w)\, f(w)\, dv(w), \quad z\in B
\end{equation*}
(for suitable  $f$). Here $dv$ stands for the volume measure in  $\mathbf{C}
^n$ normalized  in the unit ball $B$.      The  parameter $\sigma$ is a real
number. The symbol    $\left<z,w\right>$   is the standard  inner product in
$\mathbf{C}^n$. Furthermore,  $|z|=\sqrt{\left<z,z\right>}$ is the   induced
norm in  $\mathbf{C}^n$.

We will not normalize the operator $P_\sigma$  in the sense that $P_\sigma 1
= 1$, since  our aim  is to consider  not         only  the case $\sigma>-1$.

%%  THE HYPERBOLIC MEASURE
Let $d\lambda$  stand  for the invariant metric in the unit ball, i.e.,  let
\begin{equation*}
d\lambda (z) = \frac {dv(z)} {\left(1-|z|^2\right)^{n+1}}.
\end{equation*}

%% ON BESOV SPACES
Under a reasonable set of assumptions,          the Besov space $B_1$ may be
introduced as the smallest Moebius     invariant Banach space. In   the same
scale, the Bloch space is   the maximal Moebius invariant     space;     one
often writes $B_\infty$ for that space.  For this and  related results    we
refer to the Zhu  book~\cite{ZHU.BOOK}.  In particular, see Theorems 6.8 and
6.10 there. This          reference contains  also  all relevant information
concerning   the Bergman projection operator we need in this paper.

The Besov space           $B_1$ may be alternatively defined in terms of the
semi--norm.    We will consider     the following   semi--norm on $B_1$. For
$f\in B_1$ let
\begin{equation*}
\| f \|_{B_1} \  =
\sum_{|\alpha| = n+1} \int _ B  \left|\frac {\partial^{n+1} f(z)}
{\partial^\alpha  z} \right| dv(z)
\end{equation*}
(the summation is over all                                         $\alpha =
 (\alpha_1,\dots,\alpha_n)\in\mathbf{Z}^{n}_+$     which satisfy $|\alpha| =
\alpha_1 +\dots + \alpha_n =n+1$;  $\mathbf{Z}_+$  is the set of         all
non--negative integers).

%%  THE MAIN RESULT
In this paper  we  the find  the  semi--norm   $\|P_\sigma\|_{L^1(B,d\lambda)
\rightarrow  B_1}$  for $\sigma>-\left(n+1\right)$,    i.e.,    the smallest
constant $C$  such that
\begin{equation*}
\|P_\sigma f\|_{B_1}\le C\|f \|_{L^1(B,d\lambda)},\quad f\in B_1.
\end{equation*}
The  following is our main result here.

\begin{theorem}\label{TH.MAIN}
The operator $P_\sigma$ maps continuously the space  $L^1(B,d\lambda)$ onto
the Besov space $B_1$ if and    only if $\sigma>  - \left(n + 1\right)$. In
this case  we have
\begin{equation*}
\|P_\sigma \|_{ L^1(B,\lambda) \rightarrow B_1}
 = \frac{ n!\, \Gamma(n+1+\mu)}{\Gamma^2((n+ 1 + \mu)/ 2)},
\end{equation*}
where $\mu=n+1+\sigma$.
\end{theorem}

\begin{remark}
According  to this result,  for the ordinary  Bergman projection  $P=P_0$ we
have
\begin{equation*}\begin{split}
\| P\|_{L^1(B,d\lambda) \rightarrow B_1}  =  \frac {\left(2n+1\right)!}{n!}.
\end{split}\end{equation*}
In  particular,                         for    $n=1$ we have $\|P\| = 6$, as
Per\"{a}l\"{a}~\cite{PERALA.ARCH}                      has recently   showed.
\end{remark}

\begin{remark}
Results  concerning the semi--norm calculation of the operator    $c_\sigma
P_\sigma$ for $\sigma>-1$, where           $c_\sigma ={\Gamma(n+1+\sigma)}/
{(\Gamma(\sigma+1)\Gamma(n+1))}$   is a normalizing  constant, when  acting
from  the space $L^\infty (B)$ onto  the Bloch space in the unit ball,  may
be   found in~\cite{KALAJ.SCAND, KALAJ.JOT},   and   in the recent author's
preprint~\cite{MARKOVIC.PREPRINT}.   For example, for $n=1$  Per\"{a}l\"{a}
obtained
\begin{equation*}
\| P\|_{ L^\infty (B)\rightarrow B_\infty} = \frac 8\pi.
\end{equation*}
This           is the result from~\cite{PERALA.AASF.2012} which served as a
motivation for  the author  papers.
\end{remark}

\section{Some lemmas}
In order  to prove our main theorem, we need some auxiliary results. These
will be collected    in lemmas which follows.    Some of the facts we will
present in the sequel  may                          be  found in the Rudin
monograph~\cite{RUDIN.BOOK.BALL}.

\subsection{}
%% BI--HOLOMORPHIC
It is well known that bi--holomorphic mappings of $B$ onto itself,  up to
unitary transformations,  have the form
\begin{equation*}
\varphi_z(\omega)
= \frac {z-{\left<\omega,z\right>}z/{|z|^2}-(1-|z|^2)^{1/2}
(\omega-{\left<\omega,z\right>}z/{|z|^2})}
{1-\left<\omega,z\right>}
\end{equation*}
for some $z\in B$. For $z=0$ we mean $\varphi _z  = -\mathrm {Id}_{B}$.
The known  identity
\begin{equation}\label{RE.2}
\left|1-\left<z,\omega\right>\right|
\left|1-\left<z,\varphi_z(\omega)\right>\right|=1-\left|z\right|^2
\end{equation}
for                    $z,\, \omega\in B$ will be useful in the following

\begin{lemma}\label{LE.TRANS}
For every $z\in B$  there holds
\begin{equation*}\begin{split}
\int_B\, \frac {\left(1-|z|^2\right)^{a} }
{\left|1-\left<z,w\right>\right|^{n+1+a}}\, dv(w)
=  \int_B\, \frac {1}
 {\left|1-\left<z,\omega\right>\right|^{n+1-a}}\, dv( \omega),
\end{split}\end{equation*}
where $a$  is  any real number.
\end{lemma}

\begin{proof}
The       real Jacobian of $\varphi_z(\omega)$ is given by the expression
\begin{equation*}
(J_\mathbf R\varphi_z)(\omega)=
\frac{ \left(1-|z|^2\right)^{n+1}}
{\left|1-\left<z,\omega\right>\right|^{2n+2}}.
\end{equation*}

Denote the integral on  the left side of our lemma  by $J$.   Introducing
the change  of variables $w=\varphi_z(\omega)$ and using the relation for
the pull--back measure, we obtain
\[\begin{split}
J&=\int_B \frac {\left(1-|z|^2\right)^a} {\left|1-\left<z,\varphi_z(\omega)\right>\right|^{n+1+a}}
\, \frac{\left(1-|z|^2\right)^{n+1}} {\left|1-\left<z,\omega\right>\right|^{2n+2}}\, dv(\omega)
\\&=\int_{B} \frac{\left(1-|z|^2\right)^{n+1+a}}
{\left|1-\left<z,\varphi_z(\omega)\right>\right|^{n+1+a} \left|1-\left<z,\omega\right>\right|^{2n+2}}\, dv (\omega)
\\&=\int_{B}\frac{\left(\left|1-\left<z,\omega\right>\right|\left|1-\left<z,\varphi_z(\omega)\right>\right|\right)^{n+1+a}} {\left|1-\left<z,\varphi_z(\omega)\right>\right|^{n+1+a} \left|1-\left<z,\omega\right>\right|^{2n+2}}\, dv (\omega)
\\&=\int_{B}  \frac {1}{\left|1-\left<z,\omega\right>\right|^{n+1-a}}\, dv (\omega).
\end{split}\]
In third equality we have used the identity~\eqref{RE.2}
\end{proof}

\subsection{}
%% LEMMA FROM THE RUDIN BOOK
In connection with the next lemma see~\cite[Proposition~1.4.10]{RUDIN.BOOK.BALL}
as  well as~\cite{KALAJ.SCAND}.

For $z\in B,\,  c$ real, and  $t>-1$  define
\begin{equation*}
J_{c,t}(z) =
\int_{B}\, \frac{\left(1-|w|^2\right)^t}
{\left|1-\left<z,w\right>\right|^{n+1+t+c}}\,  dv(w).
\end{equation*}

\begin{lemma}\label{LE.RUDIN}
The  function $J_{c,t}(z)$  is radially symmetric and increasing in $|z|$,
since it  can be represented as
\begin{equation*}
J_{c,t}(z) = \frac{\Gamma(n+1) \Gamma(t+1)}{\Gamma(n + 1+ t)}\,
 {_2}F_1({\lambda, \lambda}, {n +1+ t}, |z|^2),
\end{equation*}
where $\lambda =  (n  + 1 + t + c)/2$, and  $\, {_2}F_1\, $ is the Gauss
hypergeometric function.

The function      $J_{c,t}(z)$ is bounded in $B$  for $c<0$. In this case
$J_{c,t}$                  extends  continuously  on  $\overline{B}$ with
\begin{equation*}
J_{c,t}(e_1)=\frac{\Gamma(n+1)  \Gamma(t+1) \Gamma(-c)}
{\Gamma^2( (n +  1   + t -  c)/2)},
\end{equation*}
where $e_1 = (1,0,\dots,0)\in\mathbf{C}^n$.
\end{lemma}

For the properties of the Gamma function and the Gauss     hypergeometric
functions we refer to~\cite{AAR.BOOK.SPECIAL}.

\section{The proof of the main theorem}
\subsection{} Since     for every  $\alpha\in\mathbf{Z}^{n}_+$ which satisfies
$|\alpha|=n+1$ we have
\begin{equation*}
\frac {\partial^{n+1} P_\sigma f(z)} {\partial^\alpha z}
=  \frac {\Gamma\left(n+1+\mu\right)} {\Gamma\left(\mu\right)}
\int_B \frac {\overline{w}^\alpha \left(1-|w|^2\right)^\sigma}
{\left(1-\left<z,w\right>\right)^{n+1+\mu}}\, f(w)\, dv(w),
\end{equation*}
we should consider the family of  operators
$\left\{Q_{\sigma,\alpha}:|\alpha|=n+1\right\}$  given by
\begin{equation*}
{Q}_{\sigma,\alpha} f(z) =  \frac {\Gamma\left(n+1+\mu\right)} {\Gamma\left(\mu\right)}
\int_B \frac { \overline{w}^\alpha\left(1-|w|^2\right)^\sigma }
{\left(1-\left<z,w\right>\right)^{n+1 +\mu}}\, f(w)\, dv(w).
\end{equation*}
For  the integer $d ={\left(2n\right)!}/({\left(n+1\right)! \left(n-1\right)!})$
(which is the  number of  all $\alpha \in\mathbf{Z}_+^{n}$ satisfying  $|\alpha|
= n+1$)  denote
\begin{equation*}
{Q}_\sigma  = (\overbrace {\dots,{Q}_{\sigma,\alpha},\dots}^d).
\end{equation*}
One  readily sees that
\begin{equation}\label{RE.NORM}\begin{split}
\|P_\sigma \| _{L^1(B,\lambda) \rightarrow B_1} &
= \|{Q}_\sigma\|_{L^1(B,\lambda)\rightarrow\bigotimes_{k=1}^d L^1(B)}
\\&= \|{Q}_\sigma ^* \|_{\bigotimes_{k=1}^d  L^\infty(B) \rightarrow L^\infty(B)}
\\& = \max_{|\alpha|=n+1} \|{Q}^*_{\sigma,\alpha}\|_{L^\infty(B)\rightarrow L^\infty(B)}.
\end{split}\end{equation}
We  will  therefore    find the conjugate operator ${Q}^*_{\sigma,\alpha}:
L^\infty(B)\rightarrow L^\infty(B)$.

The conjugate operator  of  $Q_{\sigma,\alpha}$ is
\begin{equation}\label{CONJUGATE}
{Q}_{\sigma,\alpha}^* g(z) =
\frac {\Gamma(n+1+\mu)} {\Gamma(\mu)}
\int_B  \frac {z^{\alpha}\left(1-|z|^2\right)^{\mu} }{\left(1-\left<z,w\right>\right)^{n+1+\mu}}\, g(w)\, dv(w).
\end{equation}
To see that~\eqref{CONJUGATE} is  true, let $\left<\varphi,\psi\right>_{v}$
stand for  the inner  product in  $L^2(B,dv)$. On the other hand, let
$\left<\varphi,\psi\right>_{ \lambda}$  be the inner product in
$L^2(B,d\lambda)$. Then for $f\in L^1(B,d\lambda)$ and $g\in L^\infty(B)$
we have
\begin{equation*}\begin{split}
&\left<{Q}_{\sigma,\alpha}f,g\right> _{v}=
\\&=  \frac {\Gamma(n+1+\mu)} {\Gamma(\mu)}
\int_B\left\{\int_B  \frac {\overline{w}^\alpha\left(1-|w|\right)^\sigma} {\left(1-\left<z,w\right>\right)^{n+1+\mu}}\, f(w)\, dv(w)\right\}
\overline{g(z)}\, dv(z)
\\& =  \frac {\Gamma(n+1+\mu)} {\Gamma(\mu)}
\int_B \left\{\int_B  \frac {\overline{w}^\alpha  \left(1-|w|\right)^\sigma }{\left(1-\left<z,w\right>\right)^{n+1+\mu}}\, \overline{g(z)}\, dv(z)\right\} f(w)\, dv(w)
\\&=  \frac {\Gamma(n+1+\mu)} {\Gamma(\mu)}
\int_B\left\{\int_B \frac {\overline{w}^\alpha  \left(1-|w|\right)^{\mu} }{\left(1-\left<z,w\right>\right)^{n+1+\mu}}\, \overline{g(z)} \, dv(z)\right\}f(w)\, d\lambda(w)
\\&= \frac {\Gamma(n+1+\mu)} {\Gamma(\mu)}
\int_B f(w)\left\{ \overline { \int_B \frac { {w}^\alpha  \left(1-|w|\right)^{\mu} }
{\left(1-\left<w,z\right>\right)^{n+1+\mu}}\, {g(z)} \, dv(z) }\right\}  d\lambda(w)
\\&= \left<f,{Q}^*_{\sigma,\alpha}g\right>_{ \lambda},
 \end{split}\end{equation*}
where
\begin{equation*}
{Q}_{\sigma,\alpha} ^* g(w)
= \frac {\Gamma(n+1+\mu)} {\Gamma(\mu)}
\int_B \frac { w^{\alpha}\left(1-|w|^2\right)^{\mu} }{\left(1-\left<w,z\right>\right)^{n+1+\mu}}\, g(z)\, dv(z).
\end{equation*}

Let  us now find $\|{Q}^*_{\sigma,\alpha}:L^\infty(B)\rightarrow L^\infty(B)\|$.

For      fixed $z$     (and $\alpha$) the maximum of the integral expression
in~\eqref{CONJUGATE} (regarding $g\in L^\infty(B),\, \|g\|_\infty=1$)     is
attained  for $g_z\in  L^\infty (B)$ given by
\begin{equation*}
g_z(w) =   \frac{\left|1-\left<z,w\right>\right|^{n+1 +\mu}}{(1-\overline{\left<z,w\right>})^{n+1+\mu}}.
\end{equation*}
Note  that  $\|g_z\|_\infty=1$.
Therefore, we have  (for fixed $z$ and $\alpha$)
\begin{equation*}
|{Q}_{\sigma,\alpha}^* g(z)|\le  \frac {\Gamma(n+1+\mu)} {\Gamma(\mu)}
\,  |z|^\alpha  \int_B \frac {\left(1-|z|^2\right)^{\mu} }{\left|1-\left<z,w\right>\right|^{n+1+\mu}}\, dv(w)
\end{equation*}
for  all $g\in L^{\infty}(B),\, \|g\|_\infty \le  1$.

It follows
\begin{equation*}\begin{split}
\|{Q}_{\sigma,\alpha}^*\|_{L^\infty(B)\rightarrow L^\infty(B)} &\le   \frac {\Gamma(n+1+\mu)} {\Gamma(\mu)}
\sup_{z\in B}   \int_B \frac {\left(1-|z|^2\right)^{\mu} }{\left|1-\left<z,w\right>\right|^{n+1+\mu}}\,  dv(w)
\end{split}\end{equation*}
for all $\alpha\in \mathbf{Z}^{n}_+,\,   |\alpha| = n+1$.

We transform  the last integral as
\begin{equation*}\begin{split}
\int_B \frac {\left(1-|z|^2\right)^{\mu} }{\left|1-\left<z,w\right>\right|^{n+1+\mu}}\,  dv(w)
&=  \int_B \frac {dv( \omega)} {\left|1-\left<z,\omega\right>\right|^{n+1-\mu}} \\&= J_{-\mu,0}(z)
\end{split}\end{equation*}
(see Lemma~\ref{LE.TRANS}).

Regarding the first part of Lemma~\ref{LE.RUDIN} the number $\sup_{z\in B} J_{-\mu,0} (z)$
is finite if    $-\mu<0$ i.e.,  $\sigma > - \left(n+1\right)$.  In this case, by the second
part of this lemma,  we can write
\begin{equation*}
\sup_{z\in B} J_{-\mu,0} (z)
 =  \frac{\Gamma(n+1) \Gamma (\mu)}{\Gamma^2((n+1+{\mu})/2)}.
\end{equation*}
Therefore, we have
\begin{equation*}\begin{split}
\| {Q}_{\sigma,\alpha}^*\|_{L^\infty(B)\rightarrow L^\infty(B)}
&\le \frac {\Gamma(n+1+\mu)} {\Gamma(\mu)} \frac{\Gamma(n+1)  \Gamma (\mu)}{\Gamma^2((n+1 + {\mu})/2)}
\\& = \frac{n!\,\Gamma(n+1+\mu)}{\Gamma^2 ((n+ 1 + {\mu})/2)}.
\end{split}\end{equation*}

Regarding now the relation~\eqref{RE.NORM} for $\sigma>- \left(n+1\right)$ we
obtain
\begin{equation*}
\|P_\sigma \|_{ L^1(B,\lambda) \rightarrow B_1}\le
\frac{n!\, \Gamma(n+1+\mu)}{\Gamma^2( (n+ 1+{\mu})/2)},
\end{equation*}
what gives one part of our theorem.

\subsection{}
We will now prove the reverse inequality as well as that the condition
$\sigma>-\left(n+1\right)$               is  necessary for boundedness
of $P_\sigma$ on $L^1(B,d\lambda)$.                    We also use the
relation~\eqref{RE.NORM}.

For $\varepsilon\in \left(0,1\right)$ denote
\begin{equation*}
g_\varepsilon (w) =
\frac{\left|1-\left<\varepsilon e_1,w\right>\right|^{n+1+\mu}}
{(1-\overline{\left<\varepsilon e_1,w\right>})^{n+1+\mu}}.
\end{equation*}
Then $g_\varepsilon\in L^\infty(B),\, \|g_\varepsilon\|_\infty=1$ and
for  any $|\alpha|=n+1$ we have
\begin{equation*}\begin{split}
{Q}_{\sigma,\alpha}^* g_\varepsilon(\varepsilon e_1)&
= \frac {\Gamma(n+1+\mu)} {\Gamma(\mu)}\, \varepsilon^{n+1}
\int_B \frac{\left(1-\varepsilon^2\right)^{\mu}}{\left|1-\left<\varepsilon e_1,w\right>\right|^{n+1+\mu}}\,  dv(w)
\\&= \frac {\Gamma(n+1+\mu)} {\Gamma(\mu)}\, \varepsilon ^{n+1} \int_B \frac{dv(w)}{\left|1-\left<\varepsilon e_1,w\right>\right|^{n+1-\mu}}
\\&= \frac {\Gamma(n+1+\mu)} {\Gamma(\mu)}\,  \varepsilon ^{n+1}   \,  J_{-\mu,0} (\varepsilon e_1).
\end{split}\end{equation*}
It follows
\begin{equation*}\begin{split}
\|{Q}_{\sigma,\alpha}^* \|_{L^\infty(B)\rightarrow L^\infty(B)}&\ge
\sup_{z\in B} | {Q}_{\sigma,\alpha}^* g_\varepsilon(z)|
\\&\ge\limsup_{\varepsilon\rightarrow 1} |{Q}_\sigma^* g_\varepsilon(\varepsilon e_1)|
\\&= \frac {\Gamma(n+1+\mu)} {\Gamma(\mu)}    \lim_{\varepsilon\rightarrow 1} J_{-\mu,0}(\varepsilon e_1)
\\&= \frac {\Gamma(n+1+\mu)} {\Gamma(\mu)}  \frac{\Gamma(n+1)  \Gamma (\mu)}{\Gamma^2((n+ 1+\mu)/2)}
\\& = \frac{ n!\, \Gamma(n+1+\mu)}{\Gamma^2((n+ 1+{\mu})/2)},
\end{split}\end{equation*}
only for $\sigma>- \left(n+1\right)$.

Thus, in view of~\eqref{RE.NORM} we have
\begin{equation*}
\|P_\sigma \|_{ L^1(B,\lambda) \rightarrow B_1} \ge
\frac {n!\, \Gamma(n+1+\mu)}{\Gamma^2((n+1+{\mu})/2)}
\end{equation*}
for $\sigma>- \left(n+1\right)$.

\end{document}